\documentclass[11pt,a4paper,reqno]{amsart}%
\usepackage{amsthm,amsmath,amsfonts,amssymb,amsxtra,appendix,bookmark,euscript,delarray,dsfont,color,mathrsfs}
\usepackage[latin1]{inputenc}

\theoremstyle{plain}
\newtheorem{theorem}{Theorem}[section]
\newtheorem{lemma}[theorem]{Lemma}
\newtheorem{corollary}[theorem]{Corollary}

\theoremstyle{remark}

\numberwithin{equation}{section}

\newcommand\R{{\ensuremath {\mathbb R} }}

\newcommand\N{{\ensuremath {\mathbb N} }}

\newcommand\1{{\ensuremath {\mathds 1} }}

\newcommand\nn{\nonumber}

\newcommand{\scF}{\mathscr{F}}

\newcommand{\cB}{\mathcal{B}}

\newcommand{\cE}{\mathcal{E}}

\newcommand{\cL}{\mathcal{L}}

\newcommand{\eps}{\epsilon}

\renewcommand{\epsilon}{\varepsilon}

\DeclareMathOperator{\Tr}{{\rm Tr}}

\newcommand{\supp}{{\rm supp\,}}

\renewcommand{\ge}{\geqslant}
\renewcommand{\le}{\leqslant}
\renewcommand{\geq}{\geqslant}
\renewcommand{\leq}{\leqslant}
\renewcommand{\hat}{\widehat}

\begin{document}
\title[Semiclassical Moser--Trudinger inequalities]{Semiclassical Moser--Trudinger inequalities}

\author[R. Arora]{Rakesh Arora}
\address{Department of Mathematical Sciences, Indian Institute of Technology Varanasi (IIT-BHU), Uttar Pradesh-221005, India} 
\email{rakesh.mat@iitbhu.ac.in, \,  arora.npde@gmail.com}

\author[P.T. Nam]{Phan Th\`anh Nam}
\address{Department of Mathematics, LMU Munich, Theresienstrasse 39, D-80333 Munich, and Munich Center for Quantum Science and Technology (MCQST), Schellingstr. 4, D-80799 Munich, Germany} 
\email{nam@math.lmu.de}

\author[P.T. Nguyen]{Phuoc-Tai Nguyen}
\address{Department of Mathematics and Statistics, Masaryk University, Brno, Czechia}
\email{ptnguyen@math.muni.cz}

\begin{abstract} We extend the Moser--Trudinger inequality of one function to systems of orthogonal functions. Our results are asymptotically sharp when applied to the collective behavior  of eigenfunctions of Schr\"odinger operators on bounded domains.  

\medskip

  \medskip

	
\end{abstract}


%
\maketitle

\setcounter{tocdepth}{2}
\tableofcontents

\section{Introduction}

It is well-known that the Sobolev embedding $H^1(\R^d)\subset L^\infty(\R^d)$ holds in one dimension ($d=1$), but barely fails in two dimensions ($d=2$) due to a logarithmic divergence. In several two-dimensional problems,  the  Moser--Trudinger inequality serves as a great replacement for the missing Sobolev embedding. In the simplest form, this result states that there exist universal constants $\alpha>0,C>0$ such that for every open bounded set $\Omega \subset \R^2$ we have 
\begin{equation} \label{eq:MT}
	\frac{1}{|\Omega|}\int_{\Omega} \exp\left({\alpha |u(x)|^2} \right) {\rm d} x \le C, \quad \forall u\in H_0^1(\Omega), \quad  \|\nabla u\|_{L^2(\Omega)} \le 1. 
\end{equation}
This inequality was first proved by Trudinger \cite{Trudinger}, and the optimal value $\alpha=4\pi$ was derived later by Moser \cite{Moser}. 

From the physical perspective, the Moser--Trudinger inequality, as well as the Sobolev inequality, is a quantitative form of the {\em uncertainty principle}, in which the localization in the momentum space (i.e. the boundedness of the kinetic energy) implies the absence of singularities in the configuration space in an appropriate sense. In the present paper, we will extend this inequality by combining with another fundamental law in quantum mechanics, the {\em exclusion principle}. To be precise, we aim at extensions of \eqref{eq:MT} for $N$ functions $\{u_n\}_{n=1}^N \subset H_0^1(\Omega)$ satisfying that $\{\nabla u_n\}_{n=1}^N$ are orthonormal in $L^2(\Omega)$, namely 
\begin{equation} 
	\label{assumpt:ON1} 
	\int_{\Omega} \overline{\nabla u_n(x)} \cdot \nabla u_m(x) \, {\rm d} x = \delta_{nm} = \begin{cases}1\quad  & \text{if }m=n,\\0\quad & \text{if } m\ne n\end{cases}, \quad \forall 1 \leq n,m \leq N.
\end{equation}
Let us denote the  {\em one-body density} of $\{u_n\}_{n=1}^N$ by 
\begin{equation} 
	\label{eq:def-rho} 
	\rho(x):= \sum_{n=1}^N |u_n(x)|^2.
\end{equation}
By the Hoffmann--Ostenhof inequality \cite{HofOst-77},  
$$	
\int_{\Omega} |\nabla \sqrt{\rho}(x)|^2 \, {\rm d} x \le \sum_{n=1}^N \int_{\Omega} |\nabla u_n(x)|^2 \, {\rm d} x  = N.
$$
Hence, applying \eqref{eq:MT} with $u=\sqrt{\rho/N}$, we find that for all $N\ge 1$, 
\begin{equation} \label{eq:MT-ext-trivial}
	\frac{1}{|\Omega|}\int_{\Omega} \exp \left({\frac{\alpha}{N} \rho(x)}\right) {\rm d} x \le C,
\end{equation}
with  the same constants $\alpha,C$ in \eqref{eq:MT}. Note that \eqref{eq:MT-ext-trivial} would be optimal if we only assumed the normalization $\|\nabla u_n\|_{L^2}=1$ for all $n=1,2,...,N$ (we may take $u_n=u$ for all $n$).  On the other hand, given additionally the orthogonality of $\{\nabla u_n\}_{n=1}^N$,  one may hope to improve the exponential factor $\alpha/N$ in \eqref{eq:MT-ext-trivial} when $N\to \infty$. We will confirm this by deriving improved versions of   \eqref{eq:MT-ext-trivial}  in which the $N$-dependence is {\em semiclassically optimal}. 

Let us quickly explain the relation between the bounds we are interested in and semiclassical properties of Schr\"odinger operators. In  \cite{Lieb}, Lieb proved that in dimensions $d\ge 3$, the orthogonality of  $\{\nabla u_n\}_{n=1}^N$ in $L^2(\R^d)$ implies the density bound
\begin{align}
	\label{eq:CLR}
	N = \sum_{n=1}^N \| \nabla u_n\|_{L^2(\R^d)}^2 \ge C_d \int_{\R^d} \rho(x)^{\frac {d} {d-2}} \, {\rm d} x 
\end{align}
for a universal constant $C_d>0$ (the density $\rho(x)$ is always defined by \eqref{eq:def-rho}). Note that by a well-known duality argument (see e.g. \cite{Frank-14} for an explanation), \eqref{eq:CLR} is equivalent to the Cwikel--Lieb--Rozenblum (CLR) inequality on the number of negative eigenvalues of Schr\"odinger operators: 
\begin{align}
	\label{eq:CLR-V}
	\Tr_{L^2(\R^d)} \1_{\{ -\Delta - V<0 \}} \le C_{{\rm CLR},d} \int_{\R^d} [V(x)]_+^{\frac{d}{2}} \, {\rm d}x. 
\end{align}
The significance of \eqref{eq:CLR-V} is that, up to a universal constant factor,  the right-hand side coincides with the semiclassical expression 
\begin{align}
	\label{eq:CLR-cl}
	\int_{\R^d}\int_{\R^d} \1_{\{|\xi|^2 - V(x)<0\}} \, {\rm d}\xi \, {\rm d}x = C_{{\rm cl},d} \int_{\R^d} [V(x)]_+^{\frac{d}{2}} \, {\rm d}x. 
\end{align}
It is also known that the best constant $C_{{\rm CLR},d}$ in \eqref{eq:CLR-V} is different from the semiclassical constant $C_{{\rm cl},d}$ in \eqref{eq:CLR-cl}; we refer to \cite{HKRV-22} for some recent progress towards the understanding of $C_{{\rm CLR},d}$. 

In dimension $d=2$, the CLR inequality \eqref{eq:CLR-V} fails to hold since $-\Delta-V(x)$ always has at least one negative eigenvalue if $V\ge 0$ and $V\not \equiv 0$, no matter how small the potential is. A two-dimensional alternative of the CLR inequality was given by Solomyak in \cite{Solomyak-94} where he derived an upper bound on the number of negative eigenvalues of Schr\"odinger operator $-\Delta-V(x)$ on a domain in terms of a certain Orlicz norm of $V$. In a remarkable paper \cite{FraLap-18}, Frank and Laptev showed that Solomyak's bound can be made to be independent of the domain in the case of Dirichlet boundary condition. By a duality argument, the  result in \cite{FraLap-18}  implies that, under the orthonormality assumption \eqref{assumpt:ON1}, there exist universal constants $\alpha>0,C_\alpha>0$ such that
\begin{align}\label{eq:FL}
	\frac{1}{|\Omega|}\int_\Omega \exp \left( \frac{\alpha}{\ln N} \rho(x) \right) {\rm d} x \leq C_\alpha.
\end{align}
To be precise, in \cite{FraLap-18} the authors showed that $|\Omega|^{-1} \int_\Omega A( \alpha \rho/\ln N) \le C_\alpha$ with $A(t)=e^t-t-1$, but we can easily obtain \eqref{eq:FL} with a different $\alpha$ by using $A(t)\ge e^{t/2}-C$. Alternatively, the proof of \cite[Lemma 3.1]{FraLap-18} allows to treat directly $A(t) = e^t$ without changing $\alpha$. 
On the technical level, the function $A$ was chosen in \cite{FraLap-18} to satisfy the typical assumption $\lim_{t\to 0}A(t)/t=0$ in the Orlicz space literature, but the authors did not use this condition anywhere in the proof. 

%

In the present paper, we will extend \eqref{eq:FL} to all $\alpha>0$ and give a quantitative version of \eqref{eq:FL} which highlights the semiclassical relation between the constants $\alpha$ and $C_\alpha$. This is very different from the standard Moser--Trudinger inequality \eqref{eq:MT}, where only the constant in the exponential  matters.

A lower bound of $C_\alpha$ in terms of $\alpha$ can be obtained easily by considering the eigenfunctions of the Dirichlet Laplacian $-\Delta$ on $L^2(\Omega)$. The celebrated Weyl's law \cite{Weyl} states that $-\Delta$ has compact resolvent and its eigenvalues $0<\lambda_1 < \lambda_2 \le \lambda_3 \le ...$ satisfy
\begin{align} \label{eq:Weyl}
	\lim_{n\to \infty} \frac{\lambda_n}{n} = \frac{4\pi}{|\Omega|}. 
\end{align}
Let $\{\varphi_n\}_{n=1}^\infty$ be the corresponding orthonormal eigenfunctions in $L^2(\Omega)$, namely
\begin{equation*}  \begin{cases}
		- \Delta \varphi_n \!\!\!\!\!&= \lambda_n \varphi_n \quad \,\,\text{in } \Omega, \\
		\quad \varphi_n  &=0 \qquad\quad \text{on } \partial \Omega,\\
		\|\varphi_n\|_{L^2(\Omega)} &=1, \quad \forall n=1,2,...
	\end{cases}
\end{equation*}
Then the functions 
\begin{equation}\label{eq:un-ef-Laplacian} u_n=\lambda_n^{-\frac{1}{2}}\varphi_n \in H_0^1(\Omega)
\end{equation}
satisfy the orthonormality \eqref{assumpt:ON1}. Moreover, by \eqref{eq:Weyl} we have
\begin{align*} \int_{\Omega}\rho(x) {\rm d} x &= \sum_{n=1}^N \int_{\Omega} |u_n(x)|^2 \, {\rm d} x = \sum_{n=1}^N \frac{1}{\lambda_n} \\
	&= \sum_{n=1}^N \frac{1}{n}\cdot \frac{|\Omega|}{4\pi}  \left( 1   + o(1)_{n\to \infty} \right) = \ln(N) \cdot \frac{|\Omega|}{4\pi} \left(  1  + o(1)_{N\to \infty} \right). 
\end{align*}
Consequently, applying Jensen's inequality for the convex function $t\mapsto e^{t}$, we conclude that for every constant $\alpha>0$, 
\begin{align} \label{eq:ex-semi}
	\frac{1}{|\Omega|}\int_{\Omega} \exp \left( \frac{\alpha}{\ln N} \rho(x) \right) {\rm d }x  &\geq \exp \left( \frac{\alpha}{|\Omega| \ln N}  \int_{\Omega} \rho(x)  {\rm d} x \right)  \nonumber\\
	&= \exp \left(  \frac{\alpha}{4\pi} \left( 1 + o(1)_{N\to \infty} \right)  \right).
\end{align}
Thus if \eqref{eq:FL} holds for $N$ large, then the two constants  $(\alpha,C_\alpha)$ are related  in such a way that 
$$
	C_\alpha \ge \exp \left( \frac{\alpha}{4\pi} \right).
$$
On the other hand, we will show below that it is possible to take $C_\alpha$ in \eqref{eq:FL} essentially as small as $e^{\frac{\alpha}{4\pi}}$, thus making the semiclassical nature of the relevant inequality explicit. 

We will also consider a bounded domain $\Omega\subset \R^d$ in any dimension  $d\ge 1$, where \eqref{eq:FL}  extends naturally to functions  $\{u_n\}_{n=1}^N \subset L^2(\Omega)$  satisfying 
\begin{align}\label{eq:constraint-simple-intro}
\sum_{n=1}^N |u_n\rangle \langle u_n| \le (-\Delta_{\Omega})^{-s} \quad \text{ on }L^2(\Omega)
\end{align}
with $s=d/2$, where $(-\Delta_{\Omega})^{s}$ is the fractional Laplacian on $L^2(\Omega)$ associated with Dirichlet boundary conditions (see \eqref{eq:quadratic-form-Delta-intro} below) and $(-\Delta_{\Omega})^{-s}$ denotes the inverse operator of $(-\Delta_{\Omega})^s$. On the left-hand side of \eqref{eq:constraint-simple-intro}, $| \cdot \rangle$ and $\langle \cdot |$ denote the bra and ket notations. The constraint \eqref{eq:constraint-simple-intro} can be interpreted as an exclusion principle in the energy space in the same spirit of \cite{Rumin,Frank-14}. Moreover, we may replace  $(-\Delta_{\Omega})^{s}$ by $(-\Delta_\Omega)^s -V(x)$ with a suitable potential $V:\Omega \to \R$, thus obtaining a semiclassical bound for general Schr\"odinger operators on bounded sets.  The precise statements of our results and the main ideas of the proofs are given below.

\subsection{Main results}  \label{sec:main-results}

Our first new result is the following extension of the two-dimensional Moser--Trudinger inequality to systems of orthonormal functions.

\begin{theorem}[Semiclassical Moser--Trudinger inequality] \label{thm:main1} Let $\Omega \subset \R^2$ be an open bounded set. Let  $\{u_n\}_{n=1}^N \subset H_0^1(\Omega)$ satisfy the orthonormality \eqref{assumpt:ON1}. Define $\rho(x)=\sum_{n=1}^N |u_n(x)|^2$. Then for all $N\ge 1$ and all $\eps\in (0,1/4]$ we have
	\begin{equation} 
		\label{est:4pi(1-eps)-1} \frac{1}{|\Omega|}\int_{\Omega} \exp(4\pi (1-\eps)\rho(x)) \, {\rm d} x \leq \frac{e^{8}}{\eps^4} N. 
	\end{equation} 
	Consequently, for any $\alpha>0$ and $N\geq \max\{ e^4, e^{\frac{\alpha}{4\pi} +1}  \}$, there holds
	\begin{equation} 
		\label{est:6} \frac{1}{|\Omega|}\int_{\Omega} \exp\left(\alpha \frac{\rho(x)}{\ln(N)} \right){\rm d} x \leq 	\exp \left( \frac{\alpha}{4\pi}\Big[ 1 + \frac{14 \ln(\ln(N))}{\ln(N)} \Big] \right).
	\end{equation}
\end{theorem}

Our first bound \eqref{est:4pi(1-eps)-1} can be interpreted as a two-dimensional alternative of the CLR inequality \eqref{eq:CLR}-\eqref{eq:CLR-V}, and in the case $N=1$, it gives us the Moser-Trudinger inequality \eqref{eq:MT} with an almost optimal constant $4\pi(1-\eps)$ in the exponent and an explicit upper bound $e^8/\eps^4$. Estimate \eqref{est:4pi(1-eps)-1} also implies the second bound \eqref{est:6} via the standard Jensen inequality. 

The second bound \eqref{est:6} is essentially the desired inequality \eqref{eq:FL} with $C_\alpha \approx e^{\frac{\alpha}{4\pi}}$ in the large $N$ limit.  To be precise,  from \eqref{est:6} and \eqref{eq:ex-semi} we have the following immediate consequence. 

\begin{corollary} For every open bounded set $\Omega \subset \R^2$, we have 
$$
	\lim_{N\to \infty} \left( \sup_{\{u_n\}_{n=1}^N} \int_{\Omega} \exp\left(\frac{\alpha}{\ln N} \sum_{n=1}^N |u_n(x)|^2 \right){\rm d} x \right) = \exp \left( \frac{\alpha}{4\pi} \right)
$$
where the supremum is taken over all sequences $\{u_n\}_{n=1}^N \subset H_0^1(\Omega)$ satisfying the orthonormality \eqref{assumpt:ON1}. Moreover, for the specific choice of  $\{u_n\}_{n=1}^N$  in \eqref{eq:un-ef-Laplacian} we have
$$
	\label{est:consequence-1b} 
	\lim_{N\to \infty} \left(  \int_{\Omega} \exp\left(\frac{\alpha}{\ln N} \sum_{n=1}^N |u_n(x)|^2 \right){\rm d} x \right) = \exp \left( \frac{\alpha}{4\pi} \right).
$$
\end{corollary}

\medskip

Our result can be extended to a general Schr\"odinger operator $(-\Delta_\Omega)^s -V(x)$ on $L^2(\Omega)$ with $s=d/2$ and $\Omega\subset \R^d$ being open and bounded, in any dimension $d\ge 1$. Here since $s=d/2$ is not necessarily an integer, $(-\Delta_\Omega)^s$ is defined as the fractional Laplacian on $\Omega$ with Dirichlet boundary conditions, namely it is a self-adjoint operator on $L^2(\Omega)$ given by the quadratic form 
\begin{align} \label{eq:quadratic-form-Delta-intro}
\langle u, (-\Delta_\Omega)^s u \rangle = \| u \|_{H_0^s(\Omega)}^2 =  \int_{\mathbb{R}^d} |2 \pi \xi|^{2s} |\hat{\cE_{\Omega} u}(\xi)|^2 \,{\rm d} \xi 
\end{align}
where
$\cE_{\Omega}: L^2(\Omega) \to L^2(\R^d)$ is the extension operator
$$(\cE_{\Omega} u)(x) = u(x) \quad \text{ for }x\in \Omega, \quad (\cE_{\Omega} u)(x) =0 \text { for } x\in \R^d \setminus \Omega,$$
and we use the following convention of the Fourier transform
$$
\widehat f (\xi) = \scF(f)(\xi)  := \int_{\R^d} f(x) e^{-2\pi \mathrm{i} \xi\cdot x} \, {\rm d} x.
$$
We will assume $V\in L^p(\Omega)$ for some $p>s$, which ensures that 
$$\mathcal{L}_V:=(-\Delta_\Omega)^s -V(x)$$
is well-defined by Friedrichs' method as a self-adjoint operator on $L^2(\Omega)$ with the same quadratic form domain 
$$
H_0^s(\Omega) = \overline{C_c^\infty(\Omega)}^{\|.\|_{H_0^s (\Omega)}} =  \{u : \Omega \to \mathbb{C} \,| \,  \cE_{\Omega} u \in H^s(\mathbb{R}^d)\}.
$$

Our second main result is the following generalization of Theorem \ref{thm:main1} to Schr\"odinger operators.  

\begin{theorem}[Semiclassical bound for Schr\"odinger operators] \label{thm:main-Sch} Let $d\ge 1$ and $s=d/2$. Let $\Omega \subset \R^d$ be an open bounded set. Let $V: \Omega \to \mathbb{R}$ satisfy that $V\in L^p(\Omega)$ for some $p>s$ and that 
$$
\mathcal{L}_V = (-\Delta_\Omega)^s -V(x) >0 \quad \text{ on }L^2(\Omega)
$$
where $\mathcal{L}_V$ is a self-adjoint operator on $L^2(\Omega)$ defined by Friedrichs' method with Dirichlet boundary conditions. Let  $\{u_n\}_{n=1}^N \subset H_0^s(\Omega)$ satisfy 
\begin{align} \label{eq:exclusion-Lv-intro}
\sum_{n=1}^N |u_n\rangle \langle u_n| \le \mathcal{L}_V^{-1} \quad \text { on }L^2(\Omega). 
\end{align}
Define $\rho(x)=\sum_{n=1}^N |u_n(x)|^2$. Then for all $\alpha>0$ and for all $N$ sufficiently large, we have 
\begin{equation} \label{MT-frac-1-V-intro}
		\frac{1}{|\Omega|}	\int_{\Omega}\exp \left({\alpha \frac{\rho(x)}{\ln(N)}} \right) {\rm d} x \leq \exp \left( \frac{\alpha \omega_d}{(2\pi)^d} \Big[ 1 + \frac{C}{(\ln(N))^t}   \Big]  \right).
	\end{equation}
	Here $\omega_d$ denotes the volume of the unit ball in $\R^d$, and the constants $C=C(d,p,\Omega,V)>0$ and $0<t=t(d,p) <1$ are independent of $N,\alpha$. 
\end{theorem}

In the case $V=0$, the error term $\frac{C}{(\ln(N))^t}$ in the exponent on the right-hand side of    \eqref{MT-frac-1-V-intro} can be replaced by $\frac{4(2\pi)^d}{\omega_d} \cdot \frac{\ln(\ln(N))}{\ln(N)} $; see Theorem \ref{thm:main2} for details.

\subsection{Proof strategy}

Let us quickly explain our proof strategy.

Our proof of Theorem \ref{thm:main1} is inspired by the elegant approach to the Lieb--Thirring and CLR bounds  of Rumin \cite{Rumin}, but we need to handle the log-divergence in 2D of $(-\Delta)^{-1}$ in Fourier space carefully.  Roughly speaking the main idea of this method is to decompose the kinetic operator using the layer cake representation
\begin{align}\label{eq:layer-cake}
|\xi|^2 = \int_0^\infty \1_{\{|\xi|^2 >\ell \}}\, {\rm d} \ell,
\end{align}
with notation $\1_A$ for the indicator function of a set $A$, and then control the high-momentum part $\1_{\{|\xi|^2 >\ell \}}$ by the low-momentum part $\1_{\{|\xi|^2 \le \ell \}}$ thanks to Plancherel's identity and the triangle inequality. 

When $d\ge 3$, the low-momenta part is uniformly bounded and the growth of the $L^\infty$-norm can be estimated in terms of the Fourier cut-off $\ell$ explicitly using the orthogonality assumption \eqref{assumpt:ON1} via Bessel's inequality. However, when $d=2$, the use of Bessel's inequality is insufficient due to a logarithmic divergence  -- this is the same logarithmic problem of the Sobolev inequality (namely a problem associated to the case $N=1$, for which Bessel's inequality does not help). Thus the low-momentum part has to be handled differently. In principle, we can avoid this difficulty by replacing $-\Delta$ by $-\Delta+E$ for some constant $E>0$ which creates the energy gap for low momenta; see \cite[Theorem 2.3]{Nam-19} for a result in this direction. In the present paper, we will use the boundedness of $\Omega$ to gain the necessary energy gap. The key point is that we are able to extract the full semiclassical behavior essentially from high-momenta, thus obtaining the sharp semiclassical constant in  Theorem \ref{thm:main1}.

\medskip

To obtain Theorem \ref{thm:main-Sch}, we will need to compare the operator $\mathcal{L}_V$ with $\mathcal{L}_0=(-\Delta_\Omega)^s$. From the condition $V\in L^p(\Omega)$ and $\mathcal{L}_V=\mathcal{L}_0-V>0$, it is easy to see that $\cL_V\ge \eta \cL_0$ for some constant $\eta>0$. This allows us to deduce from \eqref{eq:exclusion-Lv-intro} that 
$$
\sum_{n=1}^N|u_n\rangle \langle u_n| \le \eta^{-1} \mathcal{L}_0^{-1} \quad \text { on }L^2(\Omega). 
$$
Unfortunately, this simple bound only gives a weaker form of \eqref{MT-frac-1-V-intro},
$$		
\frac{1}{|\Omega|}	\int_{\Omega}\exp \left({\alpha \frac{\rho(x)}{\ln(N)}} \right) {\rm d} x \leq \exp \left( \frac{\alpha \eta^{-1}\omega_d}{(2\pi)^d} \Big[ 1 + \frac{C\ln(\ln(N))}{\ln(N)}   \Big]  \right),
$$
where we lose the semiclassical constant in the exponent on the right-hand side by a factor $\eta^{-1}$ which is presumably large. 
However, as we argued before, since the semiclassical behavior essentially comes from high-momenta, we may replace the resolvent estimate $\cL_V^{-1} \le \eta^{-1} \cL_0^{-1}$ by the second order expansion 
\begin{align} \label{eq:resolvent-second-intro}
\cL_V^{-1} \le (1+\eps) \cL_0^{-1} + C_\eps \cL_0^{-2}
\end{align}
for arbitrarily small $\eps>0$. To use the proof strategy of Theorem \ref{thm:main1}, which is based on the Fourier transform, we still need to relate the  operator $\cL_0=(-\Delta_\Omega)^{s}$ by the operator $(-\Delta_{\R^d})^{s}$, but it can be done by observing that the operator 
$\cB:=(-\Delta_{\R^d})^{\frac{s}{2} } \cE_{\Omega} (-\Delta_{\Omega})^{-\frac{s}{2} }$ is bounded from $L^2(\Omega)$ to $L^2(\R^d)$ with norm 1. 

In the rest of the paper, we will prove Theorem \ref{thm:main1} in Section \ref{proof-thm1} and prove Theorem \ref{thm:main-Sch} in Section \ref{proof-thm-Sch}. 

\subsection*{Acknowledgments.} P. T. Nam thanks Rupert L. Frank for helpful discussions and suggestions. This work has been prepared with the support of Czech Science Foundation, Project GA22-17403S for the three authors. P. T. Nam was partially supported by the Deutsche Forschungsgemeinschaft (DFG, German Research
Foundation) under Germany's Excellence Strategy EXC-2111-390814868. R. Arora acknowledges the partial support of the Start-up Research Grant (SRG) SRG/2023/000308, Science and Engineering Research Board (SERB), India, and Seed grant IIT(BHU)/DMS/2023-24/493.

\section{Proof of Theorem \ref{thm:main1}} \label{proof-thm1}

Our approach is inspired by a new proof of \eqref{eq:CLR} due to Rumin \cite{Rumin} who extended elegantly the Chemin--Xu proof of Sobolev inequality \cite{CheXu} to many functions by using the orthogonality \eqref{assumpt:ON1}  via Bessel's inequality in the Fourier space; see also \cite{Frank-14} for a related approach and \cite[Section 2]{Nam-19} for a  review of this method. 

Our starting point is the following identity, with the obvious extension of $u_n$ by $\mathcal{E}_\Omega u_n \in H^1(\R^2)$, 
\begin{align*}
N = \sum_{n=1}^N \int_{\Omega}|\nabla u_n(x)|^2 \, {\rm d}x &= \sum_{n=1}^N \int_{\R^2} |2\pi \xi|^2  |\widehat u_n(\xi)|^2  \, {\rm d}\xi \\
&= \sum_{n=1}^N  \int_{\R^2} \left( \int_0^\infty \1_{\{|2\pi \xi|^2 > \ell\}} \, {\rm d\ell} \right)  |\widehat u_n(\xi)|^2  \, {\rm d}\xi \\
&=  \int_{\R^2}   \left(   \int_0^\infty \sum_{n=1}^N  |u_n(x)-u_n^{\le \ell}(x)|^2 \, {\rm d}\ell \right)  \, {\rm d} \xi,
\end{align*}
where the functions $u_n^{\le \ell}\in L^2(\R^2)$ are defined via the Fourier transform
$$
\widehat{u_n^{\le \ell}}(\xi) :=  \1_{\{|2\pi \xi|^2 \le \ell\}} \widehat {u_n} (\xi). 
$$
Here we used \eqref{eq:layer-cake} together with Fubini's and Plancherel's theorems. Next, by the triangle inequality for vectors in $\mathbb{C}^N$, we have the pointwise inequality 
$$
\sum_{n=1}^N |u_n(x)-u_n^{\le \ell}(x)|^2 \ge \left|\sqrt{\rho(x)} - \sqrt{\rho^{\le \ell}(x)}\right|^2
$$
where 
$$\rho(x)= \sum_{n=1}^N |u_n (x)|^2, \quad \rho^{\le \ell} (x)=  \sum_{n=1}^N |u_n^{\le \ell}(x)|^2.$$
Thus in summary,
\begin{align}\label{eq:Rumin-00}
N \ge \int_{\R^2}    \left(   \int_0^\infty   \left|\sqrt{\rho(x)} - \sqrt{\rho^{\le \ell}(x)}\right|^2  {\rm d}\ell   \right) {\rm d} \xi. 
\end{align}  

The key ingredient needed for our proof of  Theorem \ref{thm:main1} is the fact that the $L^\infty$-norm of the density $\rho^{\le \ell}(x)$ grows only logarithmically with the cut-off $\ell$, and that we obtain the sharp semiclassical constant $1/(4\pi)$ from its growth. To be precise, we have the following bound. 

\begin{lemma}[Uniform bound of one-body density with momentum cut-off] \label{lem:uniform-low-momenta} Let $1 \leq N \in \N$ and $\{u_n\}_{n=1}^N \subset H_0^1(\Omega)$ satisfy the orthonormality \eqref{assumpt:ON1} and identify $u_n$ with its extension $\cE_{\Omega}u_n \in H^1(\R^2)$. For $0<\delta<\ell$, denote by $u_n^{\leq \delta}$ and $u_n^{\delta,\ell}$  via the Fourier transform
 $$
 \hat {u_n^{\leq \delta}}(\xi) := \1_{ \{|2\pi \xi|^{2} \leq \delta \} }(\xi)\hat {u}_n(\xi) \quad \text{and} \quad \hat {u_n^{\delta,\ell}}(\xi) := \1_{ \{\delta < |2\pi \xi|^{2} \leq \ell \} }(\xi)\hat {u}_n(\xi)
$$
and define the densities with momentum cut-off  
\begin{equation} \label{def:rho-cut-off}
	\rho^{\leq \delta}(x) := \sum_{n=1}^N |u^{\leq \delta}_n(x)|^2 \quad \text{and} \quad \rho^{\delta,\ell}(x) := \sum_{n=1}^N |u^{\delta,\ell}_n(x)|^2.
\end{equation}
Then for all $0<\delta<\ell$ and $x \in \R^2$, 
\begin{equation} \label{est:difer}
\rho^{\delta,\ell}(x) \leq \frac{1}{4\pi} \ln \left( \frac{\ell}{\delta} \right)
\end{equation}
and
\begin{equation} \label{est:sum-un-<delta} 
\rho^{\leq \delta}(x)  \leq \frac{|\Omega|}{4\pi^2 z_{0,1}^2} \delta.
\end{equation}
Here $z_{0,1} \approx 2.4048$ denotes the first zero of the Bessel function of order zero. 
\end{lemma}
\begin{proof} By Plancherel's theorem, the condition \eqref{assumpt:ON1} implies that the functions $\{|2\pi \xi|  \hat {u}_n(\xi) \}_{n=1}^N$ are orthonormal in $L^2(\R^2)$. By the inverse Fourier transform and Bessel's inequality, we obtain, for all $x \in \R^2$, 
\begin{equation*}
\begin{split}
\sum_{n=1}^N |u_n^{\delta,\ell}(x)|^2 
&= \sum_{n=1}^N \left| \int_{ \R^2  } e^{2\pi \mathrm{i} x \cdot \xi} \1_{ \{ \delta < |2\pi \xi|^{2} \leq \ell \} }(\xi)\hat {u}_n(\xi) \, {\rm d} \xi \right|^2  \\
&= \sum_{n=1}^N \left| \int_{ \R^2 } e^{2\pi \mathrm{i} x \cdot \xi} \frac{  \1_{ \{ \delta < |2\pi \xi|^{2} \leq \ell \} }}{|2\pi \xi|}  |2\pi \xi| \hat {u}_n(\xi) \, {\rm d} \xi \right|^2 \\
&\leq \int_{\R^2} \frac{  \1_{ \{ \delta < |2\pi \xi|^{2} \leq \ell \} }}{|2\pi \xi|^2} \, {\rm d}\xi  =\frac{1}{4\pi} \ln \left(\frac{\ell}{\delta} \right),
\end{split}
\end{equation*}
which implies \eqref{est:difer}.

Next we prove \eqref{est:sum-un-<delta}. For every $x \in \Omega$, define $\varphi^x(\cdot)\in L^2(\R^d)$ via the Fourier transform by
$$ \hat{\varphi^x}(\xi) := e^{-2\pi \mathrm{i} x \cdot \xi } \1_{\{  |2\pi \xi|^{2} \leq \delta \}}(\xi), \quad \xi \in \R^2.
$$ 
Since $\1_{\Omega}\varphi^x \in L^2(\Omega)$, by the Lax-Milgram theorem, we deduce that there exists a unique solution $\phi^x \in H_0^1(\Omega)$ of the problem
$$ \left\{ \begin{aligned} - \Delta \phi^x(y) &= \varphi^x(y) \quad &&\text{in } \Omega, \\
 \quad \phi^x(y) &= 0 \quad &&\text{on } \partial \Omega
\end{aligned} \right. $$  
in the weak sense that
\begin{equation} \label{weak-form} \langle \nabla \zeta, \nabla \phi^x \rangle_{L^2(\Omega)} = \langle \zeta, \varphi^x \rangle_{L^2(\Omega)}, \quad \forall \zeta \in H_0^1(\Omega).
\end{equation}
Put differently, $\phi^x(y)=((-\Delta_\Omega)_y)^{-1} \1_{\Omega}(y) \varphi^x(y)$. By choosing $\zeta=u_n$ in \eqref{weak-form} and using $\supp u_n \subset \overline \Omega$, we have
\begin{align*} 
		\rho^{\leq \delta}(x) &= \sum_{n=1}^N |u_n^{\leq \delta}(x)|^2 = \sum_{n=1}^N \left| \int_{\R^2} e^{2\pi \mathrm{i} x \cdot \xi}  \1_{\{  |2\pi \xi|^{2} \leq \delta \}}(\xi) \hat {u}_n(\xi)\, {\rm d}\xi  \right| ^2 \nn\\
		& = \sum_{n=1}^N \left| \langle  \varphi^x, u_n \rangle_{L^2(\R^2)} \right|^2  =  \sum_{n=1}^N \left| \langle  \varphi^x, u_n \rangle_{L^2(\Omega)} \right|^2 = \sum_{n=1}^N \left| \langle  \nabla \phi^x, \nabla u_n \rangle_{L^2(\Omega)} \right|^2.
\end{align*}
Since $\{\nabla u_n\}_{n=1}^N$ are orthonormal, by using the Bessel inequality (which still holds true for orthonormal vector functions $\{\nabla u_n\}_{n=1}^N$), we have
\begin{equation} \label{rho<:c}
\rho^{\leq \delta}(x)  = \sum_{n=1}^N \left| \langle \nabla \phi^x, \nabla u_n \rangle_{L^2(\Omega)} \right|^2 \leq \|\nabla \phi^x \|_{L^2(\Omega)}^2. 
\end{equation}

Next, let $\lambda_1(\Omega)$ be first eigenvalue of the Dirichlet Laplacian $-\Delta$ in $\Omega$, namely 
$$ \lambda_1(\Omega) = \inf_{u \in H_0^1(\Omega) \setminus \{0\}} \frac{\| \nabla u \|_{L^2(\Omega)}^2}{\| u \|_{L^2(\Omega)}^2}.
$$
Taking $\zeta=\phi^x$ in \eqref{weak-form} and using the definition of $ \lambda_1(\Omega) $,  we have
\begin{align*} \| \nabla \phi^x\|_{L^2(\Omega)}^2 &= \langle \phi^x, \varphi^x \rangle_{L^2(\Omega)} \leq \| \phi^x \|_{L^2(\Omega)} \| \varphi^x \|_{L^2(\Omega)}  \leq \frac{\| \nabla \phi^x \|_{L^2(\Omega)}}{\sqrt{\lambda_1(\Omega)}} \| \varphi^x \|_{L^2(\Omega)},
\end{align*} 
which implies that 
$$
\|\nabla \phi^x\|_{L^2(\Omega)}^2  \leq \frac{1}{\lambda_1(\Omega)}\| \varphi^x \|_{L^2(\Omega)}^2 = \frac{1}{\lambda_1(\Omega)} \int_{\R^2} \1_{\{  |2\pi \xi|^{2} \leq \delta \}}(\xi) {\mathrm d}\xi = \frac{1}{\lambda_1(\Omega)} \frac{\delta}{4\pi}
$$
where we  also used Plancherel's identity. Inserting the latter bound in \eqref{rho<:c} and using the Rayleigh-Faber-Krahn inequality
$$ \lambda_1(\Omega) \geq \frac{\pi}{|\Omega|}z_{0,1}^2,
$$
where $z_{0,1} \approx 2.4048$ is the first zero of the Bessel function of order zero, 
we conclude that 
$$
\rho^{\leq \delta}(x) \le \frac{1}{\lambda_1(\Omega)} \frac{\delta}{4\pi} \le 	 \frac{|\Omega|}{4\pi^2 z_{0,1}^2} \delta.
$$
The proof of Lemma \ref{lem:uniform-low-momenta} is complete. 
\end{proof}

Now we are ready to conclude Theorem \ref{thm:main1}. 

\begin{proof}[Proof of Theorem \ref{thm:main1}] Let $\eps \in (0,1/4)$ and put $\kappa=1- (1-\eps)^{\frac{1}{4}} \in (\eps/4,1)$. We choose $\delta>0$ such that
\begin{align} \label{eq:delta-choice}
\frac{|\Omega|}{4\pi^2 z_{0,1}^2} \delta=\kappa^2.
\end{align}
By  Lemma \ref{lem:uniform-low-momenta} and the triangle inequality for vectors in $\mathbb{C}^N$, we have
\begin{align*} \sqrt{\rho^{\leq \ell}(x)} 
\leq \sqrt{\rho^{\leq \delta}(x)} + \sqrt{\rho^{\delta,\ell}(x)} \le \kappa + \sqrt{ \frac{1}{4\pi} \ln \left( \frac{\ell}{\delta} \right) }
\end{align*}
for all $0<\delta<\ell$. Inserting this into \eqref{eq:Rumin-00} we get 
\begin{equation} \label{rhorho}
N \geq \int_{\Omega} \int_\delta^\infty   \left[ \sqrt{\rho(x)} - \kappa - \sqrt{\frac{1}{4\pi}\ln\left( \frac{\ell}{\delta} \right) }  \right]_+^2 {\rm d}\ell \, {\rm d}x. 
\end{equation}
For all $x \in \Omega$ satisfying $\rho(x) > 1$, we have
\begin{align*}
& \int_\delta^\infty   \left[ \sqrt{\rho(x)} -  \kappa - \sqrt{\frac{1}{4\pi}\ln\left( \frac{\ell}{\delta} \right) }  \right]_+^2 {\rm d}\ell \\
&\geq \int_\delta^\infty   \left[ (1-\kappa) \sqrt{\rho(x)} - \sqrt{\frac{1}{4\pi}\ln\left( \frac{\ell}{\delta} \right) }  \right]_+^2 {\rm d}\ell \\	
&= \delta\int_1^\infty   \left[ (1-\kappa) \sqrt{\rho(x)} - \sqrt{\frac{1}{4\pi}\ln(\ell) }  \right]_+^2 {\rm d}\ell \\
&\geq \delta\int_1^{\exp(4\pi(1-\kappa)^4 \rho(x))}   \left[ (1-\kappa) \sqrt{\rho(x)} - (1-\kappa)^2\sqrt{\rho(x)} \right]_+^2 {\rm d}\ell \\
&\geq \delta (1-\kappa)^2  \kappa^2 \rho(x)  \left[\exp \left( 4\pi(1-\kappa)^4\rho(x) \right) -1 \right] \\
&\geq  \frac{\delta \kappa^{2}}{2}  \exp\left( 4\pi(1-\eps)\rho(x) \right), 
\end{align*}
where we used $(1-\kappa)^4 =1-\eps \ge 3/4$. Inserting this in \eqref{rhorho}, we find that 
\begin{equation} \label{est:4pi(1-eps)-2} 
\int_{ \{x \in \Omega: \rho(x) > 1 \} } \exp\left( 4\pi(1-\eps)\rho(x) \right) {\rm d}x \leq \frac{2N}{\delta \kappa^2}.
\end{equation}
Putting back the choice of $\delta$ in \eqref{eq:delta-choice} and combining with the obvious bound 
$$
\int_{ \{x \in \Omega: \rho(x) \leq 1 \} } \exp\left( 4\pi(1-\eps)\rho(x) \right) {\rm d}x \leq \exp\left( 4\pi \right) |\Omega|
$$
we arrive at 
 $$
 \frac{1}{|\Omega|}\int_{\Omega}\exp(4\pi (1-\eps)\rho(x))  {\rm d} x \leq \frac{N}{2\pi^2 z_{0,1}^2 \kappa^4} + \exp(4\pi)  \le \frac{e^{8} N } {\eps^4}
 $$
which is  \eqref{est:4pi(1-eps)-1}. In the last inequality we have used $\kappa \ge \eps/4$, $\eps\le 1/4$, and  $z_{0,1} \approx 2.4048$. 
 
Finally we choose $\eps=1/\ln(N) \le 1/4$ in \eqref{est:4pi(1-eps)-1}. Since $N \geq \exp(\frac{\alpha}{4\pi} +1)$, it follows that 
$$\beta=\frac{\alpha}{4\pi (\ln(N)-1)}\le 1.$$
Hence, by Jensen's inequality we obtain 
\begin{align*}
	&\frac{1}{|\Omega|}\int_{\Omega} \exp\left(\alpha \frac{\rho(x)}{\ln(N)} \right){\rm d} x 
	\leq \left( \frac{1}{|\Omega|} \int_{\Omega}\exp\left(4\pi (1-\eps)\rho(x)\right)  {\rm d} x\right)^\beta \\
	&\leq \left(\frac{e^{8}}{\eps^4} N  \right)^\beta= \exp \left( \frac{\alpha}{4\pi} \Big[ 1 + \frac{\ln (\ln N)}{\ln N} \cdot \frac{\ln N}{\ln N-1} \cdot \frac{9+ 4\ln(\ln(N))}{\ln (\ln N)}  \Big]  \right) \\
	&\le  \exp \left( \frac{\alpha}{4\pi} \Big[ 1 + \frac{\ln (\ln N)}{\ln N} \cdot \frac{4}{3} \cdot \frac{9+ 4\ln(4))}{\ln (4)}  \Big]  \right) \\
	& \leq \exp \left( \frac{\alpha}{4\pi} \Big[ 1 + \frac{14\ln(\ln(N))}{\ln(N)}  \Big]  \right),
\end{align*}	
which is \eqref{est:6}. The proof of Theorem \ref{thm:main1} is complete.
\end{proof}

\section{Proof of Theorem \ref{thm:main-Sch}} \label{proof-thm-Sch}

In this section we will first extend Theorem \ref{thm:main1} to all dimensions,  then prove the resolvent estimate \eqref{eq:resolvent-second-intro}, and finally  conclude Theorem \ref{thm:main-Sch}. 

\subsection{Semiclassical Moser--Trudinger inequality for all dimensions}

In this subsection we will prove the following extension of Theorem  \ref{thm:main1}.

	\begin{theorem}[Semiclassical Moser--Trudinger inequality for all $d$] \label{thm:main2}
	  Let $d\ge 1$, $s=d/2$, and $\Omega \subset \R^d$ be an open bounded set. Let $\{u_n\}_{n=1}^N \subset H_0^s(\Omega)$ satisfy that  
	\begin{equation} \label{braket-0} \sum_{n=1}^N   |u_n \rangle \langle u_n |  \leq (-\Delta_{\Omega})^{-s} \quad \text{on } L^2(\Omega).
	\end{equation}
	Define $\rho(x)=\sum_{n=1}^N |u_n(x)|^2$. Then for all $\alpha>0$ and $N \geq \max\{e^{4}, \exp(\frac{\alpha \omega_d}{(2\pi)^d} +1)\}$,  
we have
	\begin{equation*} \label{MT-frac-1}
		\frac{1}{|\Omega|}	\int_{\Omega}\exp \left({\alpha \frac{\rho(x)}{\ln(N)}} \right) {\rm d} x \leq \exp \left( \frac{\alpha \omega_d}{(2\pi)^d} \Big[ 1 +\frac{4(2\pi)^d}{\omega_d} \cdot \frac{\ln(\ln(N))}{\ln(N)}   \Big]  \right).
	\end{equation*}
	Here $\omega_d$ denotes the volume of the unit ball in $\R^d$.
\end{theorem}

Recall that the Dirichlet fractional Laplacian $(-\Delta_\Omega)^s$ is defined via the quadratic form \eqref{eq:quadratic-form-Delta-intro}, which in particular satisfies  \begin{equation} \label{eq:normu}
	\|(-\Delta_{\Omega})^{\frac{s}{2}} u\|_{L^2(\Omega)} = 	\|(-\Delta_{\R^d})^{\frac{s}{2}} (\cE_{\Omega} u)\|_{L^2(\R^d)}, \quad \forall u \in H_0^s(\Omega).
\end{equation}

Let us start by recalling the spectral gap of this operator.

\begin{lemma} \label{est-GN-frac}
Assume $d \in \mathbb{N}$, $s=d/2$ and $\Omega \subset \R^d$ is a bounded set. Then 
	\begin{equation}  \label{uL2-1}
		\|u\|_{L^2(\Omega)}^2 \leq \Lambda_d^4 |\Omega| \|(-\Delta_{\Omega})^{\frac{s}{2}} u\|_{L^2(\Omega)}^2, \quad \forall u \in H_0^s(\Omega),
	\end{equation}	
where 
\begin{equation} \label{Lambdad} \Lambda_d:=\frac{1}{(4\pi)^{\frac{d}{8}}} \left(  \frac{\Gamma(\frac{d}{4})}{\Gamma(\frac{3d}{4})}\right)^{\frac{1}{2}} \left( \frac{\Gamma(d)}{\Gamma(\frac{d}{2})} \right)^{\frac{1}{4} },
\end{equation}
with $\Gamma$ being the usual Gamma function.
\end{lemma}
\begin{proof}
By the H\"older inequality, we have
\begin{equation} \label{est:HolderL2L4}
\| u \|_{L^2(\Omega)} \leq |\Omega|^{\frac{1}{4}}\| u \|_{L^4(\Omega)} = |\Omega|^{\frac{1}{4}}\| \cE_{\Omega} u \|_{L^4(\R^d)}.	
\end{equation}
On the other hand, we have the Gagliardo-Nirenberg inequality
\begin{align} \label{est:frac-GN-ineq}
\| \cE_{\Omega} u \|_{L^4(\R^d)} &\leq 	C_{\rm GN} \| (-\Delta_{\R^d})^{\frac{s}{2}} (\cE_{\Omega} u) \|_{L^2(\R^d)}^{\frac{1}{2}} \| \cE_{\Omega} u \|_{L^2(\R^d)}^{\frac{1}{2}}\nn\\
&\le \Lambda_d \| (-\Delta_{\Omega})^{\frac{s}{2}} u \|_{L^2(\Omega)}^{\frac{1}{2}} \| u \|_{L^2(\Omega)}^{\frac{1}{2}}
\end{align}
with $\Lambda_d$ given in \eqref{Lambdad}. Here we used \eqref{eq:normu} and the upper bound $C_{\rm GN} \le \Lambda_d$  on the Gagliardo-Nirenberg optimal constant from  \cite[Proposition 5.9]{MorPiz} (with $j=0$, $r=4$, $n=s$, $\vartheta=1/2$).  Inserting \eqref{est:frac-GN-ineq} in \eqref{est:HolderL2L4} gives \eqref{uL2-1}. 
\end{proof}

Moreover, the following relation between $(-\Delta_{\Omega})^s$ and $(-\Delta_{\R^d})^s$ will be helpful. 

\begin{lemma} \label{cB1} For every $d\in \mathbb{N}$ and $s>0$, the  operator 
$$\cB:=(-\Delta_{\R^d})^{\frac{s}{2} } \cE_{\Omega} (-\Delta_{\Omega})^{-\frac{s}{2} }: L^2(\Omega)\to L^2(\R^d) $$ 
is bounded and $\|\cB\|_{\mathrm{op} }=1$.
\end{lemma}
\begin{proof} By Lemma \ref{est-GN-frac} we have $(-\Delta_\Omega)^{s} \ge C_\Omega >0$ on $L^2(\Omega)$. Combining with the fact that $H_0^s(\Omega)$ is the quadratic form domain of $(-\Delta_\Omega)^{s}$, we find that $(-\Delta_\Omega)^{\frac{s}{2}}: H_0^s (\Omega)\to L^2(\Omega)$ is  bijective, namely $(-\Delta_\Omega)^{-\frac{s}{2}}: L^2(\Omega) \to H_0^s (\Omega)$ is bijective. Hence, for every $v\in L^2(\Omega)$, by using \eqref{eq:normu} with $u=(-\Delta_\Omega)^{-\frac{s}{2}}v\in H_0^s(\Omega)$ we get
$$
\|\cB v\|_{L^2(\R^d)} = \|(-\Delta_{\R^d})^{\frac{s}{2} } (\cE_{\Omega} u)\|_{L^2(\R^d)} = \|(-\Delta_{\Omega})^{\frac{s}{2} }  u\|_{L^2(\Omega)}=\|v\|_{L^2(\Omega)}. 
$$
Therefore $\cB:L^2(\Omega) \to L^2(\R^2)$ is a bounded operator and $\|\cB\|_{\mathrm{op} }=1$.
\end{proof}

Now we establish a version of Lemma \ref{lem:uniform-low-momenta} in the nonlocal context.

\begin{lemma} \label{lem:uniform-low-momenta-frac} Assume $d \in \mathbb{N}$, $s=d/2$ and $\Omega \subset \R^d$ is a bounded set. Let $\{u_n\}_{n=1}^N \subset H_0^s(\Omega)$ satisfy \eqref{braket-0} and identify $u_n$ with $\mathcal{E}_\Omega u_n \in H^s(\R^d)$. For $0<\delta<\ell$, let 
\begin{equation} \label{un<l-frac} \hat {u_n^{\leq \delta}}(\xi) := \1_{ \{|2\pi \xi|^{2s} \leq \delta \} }(\xi)\hat {u}_n(\xi) , \quad \hat {u_n^{\delta,\ell}}(\xi) := \1_{ \{\delta < |2\pi \xi|^{2s} \leq \ell \} }(\xi)\hat {u}_n(\xi), 
\end{equation}
and define $\rho^{\leq \delta}$, $\rho^{\delta,\ell}$  as in \eqref{def:rho-cut-off}. Then for all $x \in \R^d$,
	\begin{equation} \label{est:difer-frac}
		\rho^{\delta,\ell}(x) \leq  \frac{\omega_d}{(2\pi)^d} \ln \left( \frac{\ell}{\delta} \right)
	\end{equation}
and
\begin{equation} \label{est:sum-un-<delta-frac} \rho^{\leq \delta}(x)  \leq \Lambda_d^4  \frac{\omega_d}{(2\pi)^d} |\Omega| \delta,
\end{equation}
where $\omega_d$ denotes the volume of the unit ball in $\R^d$ and $\Lambda_d$ is given in \eqref{Lambdad}.
\end{lemma}
\begin{proof} Denote 
$$
v_n := (-\Delta_{\R^d})^{\frac{s}{2}} u_n = \cB (-\Delta_{\Omega})^{\frac{s}{2}} u_n \in L^2(\R^d)
$$
with $\cB=(-\Delta_{\R^d})^{\frac{s}{2} } \cE_{\Omega} (-\Delta_{\Omega})^{-\frac{s}{2} }:L^2(\Omega)\to L^2(\R^d)$ as in Lemma \ref{cB1}. Since $\|\cB\|_{\rm op} = 1$, we deduce from the assumption \eqref{braket-0} that
$$
\sum_{n=1}^N |v_n \rangle \langle v_n| =  \cB (-\Delta_{\Omega})^{\frac{s}{2}} \left( \sum_{n=1}^N |u_n\rangle \langle u_n| \right) (-\Delta_{\Omega})^{\frac{s}{2}} \cB^* \le 1  
$$
on $L^2(\R^d)$, where $\cB^*$ is the adjoint operator of $\cB$.  Therefore, we obtain \eqref{est:difer-frac} by Bessel's inequality:
\begin{align*} 
\rho^{\delta,\ell}(x) &=\sum_{n=1}^N  |u_n^{\delta,\ell}(x)|^2   = \sum_{n=1}^N   \left| \int_{\R^d} e^{2\pi \mathrm{i} x \cdot \xi} \1_{ \{ \delta   \leq |2\pi \xi|^{2s}  < \ell \} } \widehat{u}_n (\xi) \, {\rm d}\xi \right|^2 \\
&=\sum_{n=1}^N   \left| \int_{\R^d} e^{2\pi \mathrm{i} x \cdot \xi} \1_{ \{ \delta   \leq |2\pi \xi|^{2s}  < \ell \} } |2\pi \xi|^{-s} \widehat{v}_n (\xi) \, {\rm d}\xi \right|^2 \\
&\le   \int_{\R^d}  \1_{ \{ \delta   \leq |2\pi \xi|^{2s}  < \ell \} } |2\pi \xi|^{-2s} \, {\rm d}\xi =\frac{\omega_d}{(2\pi)^d} \ln \left(\frac{\ell}{\delta} \right). 
\end{align*}

Next we prove \eqref{est:sum-un-<delta-frac}. For $x \in \Omega$, define $\varphi^x(\cdot)\in L^2(\R^d)$ via the Fourier transform by
$$
\hat{\varphi^x}(\xi) := e^{-2\pi \mathrm{i} x \cdot \xi } \1_{\{  |2\pi \xi|^{2s} \leq \delta \}}(\xi), \quad \xi \in \R^d,
$$
and define $\phi^x \in H_0^s(\Omega)$ by 
$$
\phi^x(y) := (-\Delta_{\Omega})^{-\frac{s}{2}} \1_{\Omega}(y) \varphi^x(y) .
$$

Then using $\supp u_n \subset \overline{\Omega}$ we can write
\begin{align} \label{rho<-frac-1} 
\rho^{\leq \delta}(x)  &= \sum_{n=1}^N \left| \langle  \varphi^x , u_n\rangle_{L^2(\R^d)} \right|^2 =  \sum_{n=1}^N \left| \langle  \varphi^x , u_n\rangle_{L^2(\Omega)} \right|^2 \\
&= \sum_{n=1}^N \left| \langle (-\Delta_{\Omega})^{\frac{s}{2} } \phi^x ,  u_n\rangle_{L^2(\Omega)} \right|^2 = \sum_{n=1}^N \left| \langle  \phi^x , (-\Delta_{\Omega})^{\frac{s}{2} } u_n \rangle_{L^2(\Omega)} \right|^2 . \nn
\end{align}
By the assumption \eqref{braket-0} we have 
\begin{align} \label{eq:un-ONF-D}
\sum_{n=1}^N |  (-\Delta_{\Omega})^{\frac{s}{2} }  u_n \rangle \langle  (-\Delta_{\Omega})^{\frac{s}{2} }  u_n| \le 1 \quad \text{ on }L^2(\Omega). 
\end{align}
Therefore, from \eqref{rho<-frac-1} and Bessel's inequality, we can bound
\begin{equation} \label{rho<-frac-2} \rho^{\leq \delta}(x)= \sum_{n=1}^N \left| \langle  \phi^x , (-\Delta_{\Omega})^{\frac{s}{2} } u_n \rangle_{L^2(\Omega)} \right|^2  \leq \| \phi^x \|_{L^2(\Omega)}^2.
\end{equation}
On the other hand, we have
$$
\| \phi^x \|_{L^2(\Omega)}^2 \le \Lambda_d^4 |\Omega| \| (-\Delta_{\Omega})^{\frac{s}{2} } \phi^x \|_{L^2(\Omega)} 
$$
by Lemma \ref{est-GN-frac} and 
\begin{align*}
\| (-\Delta_{\Omega})^{\frac{s}{2} } \phi^x \|_{L^2(\Omega)} &=  \| \1_\Omega \varphi^x \|_{L^2(\Omega)}^2 \le \| \1_\Omega \varphi^x \|_{L^2(\R^d)}^2 \nn\\
&= \int_{\R^d} \1_{\{ |2\pi \xi|^{2s} \leq \delta \} }(\xi) \, {\rm d}\xi = \frac{\omega_d}{(2\pi)^d} \delta
\end{align*}
by Plancherel's theorem. Inserting these bounds in \eqref{rho<-frac-2} gives \eqref{est:sum-un-<delta-frac}. 
\end{proof}

Now we are ready to conclude Theorem \ref{thm:main2}. 

\begin{proof}[Proof of Theorem \ref{thm:main2}] 	Let $\eps \in (0,1/4]$, $\kappa=1- (1-\eps)^{\frac{1}{4}}$ and choose $\delta>0$ such that
$$
\Lambda_d^4  \frac{\omega_d}{(2\pi)^d} |\Omega| \delta =\kappa^2.
$$
Using an argument similar to the one leading to \eqref{rhorho}, together with \eqref{est:difer-frac} and \eqref{est:sum-un-<delta-frac},  we obtain 
	\begin{equation*} \begin{aligned}
			N \geq \int_{\Omega} \int_\delta^\infty   \left[ \sqrt{\rho(x)} -  \kappa - \sqrt{\frac{\omega_d}{(2\pi)^d}\ln\left( \frac{\ell}{\delta} \right) }  \right]_+^2 {\rm d}\ell \, {\rm d}x. 
	\end{aligned}
\end{equation*}
	On the domain $\{x \in \Omega: \rho(x)>1\}$, by using an argument similar to the one leading to \eqref{est:4pi(1-eps)-2}, we deduce that 
		$$ 
		\int_{ \{x \in \Omega: \rho(x) > 1 \} } \exp\left( \frac{(2\pi)^d}{\omega_d}(1-\eps)\rho(x) \right) {\rm d}x \leq \frac{2N}{\delta \kappa^2}.
		$$
Combining with the obvious bound for $\{x \in \Omega: \rho(x)\le 1\}$,  
	$$
		\int_{ \{x \in \Omega: \rho(x) \leq 1 \} } \exp\left( \frac{(2\pi)^d}{\omega_d} (1-\eps)\rho(x) \right) {\rm d}x \leq \exp\left( \frac{(2\pi)^d}{\omega_d} \right) |\Omega|,
	$$
	we find that \begin{equation}\label{est:4pi(1-eps)-7-frac}
		\begin{aligned}
			\frac{1}{|\Omega|}\int_{ \Omega } \exp\left( \frac{(2\pi)^d}{\omega_d}(1-\eps)\rho(x) \right) {\rm d}x 
			&\leq \frac{2\Lambda_d^4 \omega_d N} {(2\pi)^d\kappa^4} + \exp\left( \frac{(2\pi)^d}{\omega_d} \right)\\
			&\leq \left( \frac{2 \Lambda_d^4 N}{\kappa^4 (te^t)_{|t=2\sqrt{\pi}}} + 1 \right) \exp\left( \frac{(2\pi)^d}{\omega_d} \right) \\
			&\leq \frac{22} {\eps^4} N \exp\left( \frac{(2\pi)^d}{\omega_d} \right).
	\end{aligned}\end{equation}
	Here we have used $(2\pi)^d/\omega_d \geq 2\sqrt{\pi}$, $\Lambda_d < \frac{3}{2}$, $\kappa\ge \eps/4$ and $\eps \le 1/4$. This bound holds for all $N\ge 1$ and $\eps\in (0,1/4]$.
	
	In particular, for every $\alpha>0$, if $N\ge \max\{e^4, e^{\frac{\alpha \omega_d}{(2\pi)^d} +1}\}$, then we can choose $\eps=1/\ln(N)\le 1/4$ and use Jensen's inequality with 
	$$
	\beta=\frac{\alpha \omega_d}{(2\pi)^d (\ln (N) -1)} \le 1
	$$
	to deduce from 	\eqref{est:4pi(1-eps)-7-frac} that 
	\begin{align*}
	&\frac{1}{|\Omega|}\int_{\Omega} \exp\left(\alpha \frac{\rho(x)}{\ln(N)} \right){\rm d} x \leq \left( \frac{1}{|\Omega|}\int_{\Omega}\exp\left(\frac{(2\pi)^d}{\omega_d} (1-\eps)\rho(x)\right)  {\rm d} x\right)^\beta\\
	&\leq \left( \frac{22} {\eps^4} N \exp\left( \frac{(2\pi)^d}{\omega_d} \right) \right)^\beta \\
	&= \exp\left( \frac{\alpha \omega_d}{(2\pi)^d} \Big[ 1 +  \frac{(2\pi)^d}{\omega_d} \cdot \frac{\ln (\ln N)}{\ln(N)-1}  \cdot \frac{1+ \frac{\omega_d}{(2\pi)^d} (1+\ln(22) + 4 \ln ( \ln N))}{\ln( \ln N)} \Big]  \right) \\
	&\le \exp\left( \frac{\alpha \omega_d}{(2\pi)^d} \Big[ 1 +  \frac{(2\pi)^d}{\omega_d} \cdot \frac{\ln (\ln N)}{\ln(N)} \cdot \frac{4}{3}  \cdot \frac{1+ \frac{1}{2\sqrt{\pi}} (1+\ln(22) + 4 \ln (4))}{\ln( 4)} \Big] \right) \\
	&\leq \exp \left( \frac{\alpha \omega_d}{(2\pi)^d} \Big[ 1 +\frac{4(2\pi)^d}{\omega_d}  \frac{\ln(\ln(N))}{\ln(N)}   \Big]  \right).
\end{align*}	
Here in the second last estimate we used again $(2\pi)^d/\omega_d \geq 2\sqrt{\pi}$ and $\ln(N) \ge \ln(4)$.  The proof of Theorem \ref{thm:main2} is complete.
\end{proof}

\subsection{Resolvent estimate} In this subsection, we consider the Schr\"odinger operator
$$
\cL_V=(-\Delta_\Omega)^s-V(x)\quad \text{on} \quad L^2(\Omega)
$$
with Dirichlet boundary conditions on a bounded domain $\Omega \subset \R^d$, $s=d/2$. We will compare it with the fractional Dirichlet Laplacian $\cL_0=(-\Delta_\Omega)^s$. 

Under the assumption $V\in L^p(\Omega)$ for some $p>s$, by fractional Sobolev's inequality we have the quadratic form estimate 
\begin{align} \label{eq:Sobolev-2D-Lp}
|V(x)| \le \eps (-\Delta_\Omega)^s + C \eps^{-q} \quad \text{on}\quad L^2(\Omega)
\end{align}
for all $\eps>0$, with a constant $C=C(d,p,\|V\|_{L^p})>0$ and $q=q(d,p)>0$.  Consequently,
\begin{align} \label{eq:Sobolev-LV-L0}
\cL_V \ge (1-\eps)\cL_0- C\eps^{-q},
\end{align}
and hence $\cL_V$ is well-defined as a self-adjoint operator by Friedrichs' method, with the same quadratic form $H_0^s(\Omega)$ of $\cL_0$. Moreover, by Sobolev's compact embedding $H_0^s(\Omega)\subset L^2(\Omega)$, it follows that $\cL_V$ has compact resolvent. Consequently, if $\cL_V>0$, then $\cL_V\ge E_0>0$ on $L^2(\Omega)$ for a constant $E_0>0$ depending on $\Omega,V$ (the lowest eigenvalue of $\cL_V$). Combining with \eqref{eq:Sobolev-LV-L0} we arrive at the first bound
\begin{align} \label{eq:Sobolev-2D-Lp-simple}
\cL_V \ge \eta \cL_0\quad \text{on}\quad L^2(\Omega)
\end{align}
for some $\eta=\eta(d,p,\Omega, V)>0$. This guarantees that 
\begin{align} \label{eq:Lv--L0-simple}
\cL_V^{-1}\le \eta^{-1}\cL_0^{-1}\quad \text{ on }L^2(\Omega).
\end{align}
However, the latter bound is insufficient for our purpose as $\eta^{-1}$ can be rather large. We will need the following refinement. 

\begin{lemma} \label{lem:LV-L0} Let $d\ge 1$, $s=d/2$ and $\Omega\subset \R^d$ be an open bounded set. Assume that $V\in L^p(\Omega)$ for some $p>s$ and that $\cL_V=(-\Delta_\Omega)^s-V(x)>0$ on $L^2(\Omega)$. Then 
\begin{align} \label{eq:resolvent-1}
\cL_V^{-1} \le (1+\eps) \cL_0^{-1} + C\eps^{-q} \cL_0^{-2}\quad \text{on}\quad L^2(\Omega)
\end{align}
for all $\eps>0$, where $C=C(d,p,\Omega,V)>0$ and $q=q(d,p)>0$. 
\end{lemma}

\begin{proof} Note that $\cL_V \ge \cL_{V_+}$, with $V_+=\max(V,0)$, and hence $\cL_V^{-1} \le \cL_{V_+}^{-1}$. Therefore, to prove \eqref{eq:resolvent-1} we may assume that $V\ge 0$. By the resolvent formula and the Cauchy--Schwarz inequality 
\begin{align}\label{eq:resolvent-1a}
\cL_V^{-1} - \cL_0^{-1}  &=  \cL_V^{-1} V \cL_0^{-1} =  \cL_0^{-1} V \cL_V^{-1}\nonumber\\
&\le \eps \cL_V^{-1} V \cL_V^{-1} + \eps^{-1} \cL_0^{-1} V \cL_0^{-1}, \quad \forall \varepsilon>0. 
\end{align}
Using again \eqref{eq:Sobolev-2D-Lp} we get
$$
\cL_0^{-1} V \cL_0^{-1} \le \delta \cL_0^{-1} + C\delta^{-q} \cL_0^{-2}, \quad \forall \delta>0.
$$
Moreover, combing fractional Sobolev embedding (see, e.g. \cite{MorPiz}) with \eqref{eq:Lv--L0-simple}, we have
\begin{align}\label{eq:resolvent-1c}
\cL_V^{-1} V \cL_V^{-1} \leq C  \cL_V^{-1} \cL_0 \cL_V^{-1}  \leq C  \cL_V^{-1} \le C \cL_0^{-1},
\end{align}
where $C=C(d,p,\Omega,\| V\|_{L^p})$.

In summary, from \eqref{eq:resolvent-1a} -- \eqref{eq:resolvent-1c} we conclude that
\begin{align}\label{eq:resolvent-1d}
\cL_V^{-1} - \cL_0^{-1}  \leq C (\eps + \delta \eps^{-1})  \cL_0^{-1} + C\eps^{-1} \delta^{-q} \cL_0^{-2}. 
\end{align}
Choosing $\delta=\eps^{2}$ in \eqref{eq:resolvent-1d}, we conclude \eqref{eq:resolvent-1} (the power $q>0$ may change, but it depends only on $d$ and $p$).
\end{proof}

\subsection{Conclusion of Theorem \ref{thm:main-Sch}} \label{sec:general}

Now we are able to prove Theorem \ref{thm:main-Sch}.

\begin{proof} In the following the values of $C=C(d,p,\Omega,V)>0$ and $q=q(d,p)>0$ may change from line to line. 

From the condition \eqref{eq:exclusion-Lv-intro} and the lower bound \eqref{eq:Sobolev-LV-L0},  we have
\begin{align} \label{eq:final-thm-0}
N \geq \sum_{n=1}^N \langle u_n, \cL_V u_n\rangle \ge  (1-\eps) \sum_{n=1}^N \langle u_n, \cL_0 u_n\rangle - C \eps^{-q} \int_{\Omega} \rho(x)\, {\rm d}x.
\end{align}
For $\ell>\delta>0$, define $u_n^{\leq \delta}, u_n^{\delta,\ell},\rho^{\le \delta}, \rho^{\delta,\ell}$ as in \eqref{un<l-frac} and \eqref{def:rho-cut-off}. Then using the same argument leading to \eqref{rhorho}, we have
\begin{align} \label{eq:final-thm-1}
\sum_{n=1}^N \langle u_n, \cL_0 u_n\rangle \ge \int_{\Omega} \int_\delta^\infty   \left[ \sqrt{\rho(x)} -  \sqrt{\rho^{\le \delta}(x)} - \sqrt{\rho^{\delta,\ell}(x)}  \right]_+^2 {\rm d}\ell \, {\rm d}x. 
\end{align}

Next, let us derive uniform bounds for  $\rho^{\le \delta}(x)$ and $\rho^{\delta,\ell}(x)$. 

\medskip
\noindent
{\bf Upper bound for $\rho^{\le \delta}(x)$.} Using the simple bound \eqref{eq:Lv--L0-simple}, we deduce from the condition  \eqref{eq:exclusion-Lv-intro} that 
$$
\sum_{n=1}^N |u_n\rangle \langle u_n| \le \eta^{-1} (-\Delta_{\Omega})^{-s} \quad \text{on } L^2(\Omega).
$$
Then we can proceed the same argument leading to \eqref{est:sum-un-<delta-frac} and obtain
\begin{equation} \label{est:sum-un-<delta-frac-V} \rho^{\leq \delta}(x)  \leq C |\Omega| \delta. 
\end{equation}

\medskip
\noindent
{\bf Upper bound for $\rho^{\delta,\ell}(x)$.} This part is more difficult since we want to keep the semiclassical constant. Note that the result in Lemma \ref{cB1} can be rewritten as the operator inequality
$$
\cE_\Omega (-\Delta_{\Omega})^{-s} \cE_\Omega \le (-\Delta_{\R^d})^{-s} \quad \text{ on  } L^2(\R^d). 
$$
The same bound holds with $s$ replaced by $2s$ (in fact Lemma \ref{cB1} holds for all $s>0$). Combining with the resolvent estimate in Lemma \ref{lem:LV-L0}, we find that
\begin{align*}
\cE_\Omega \cL_V^{-1} \cE_\Omega &\le (1+\eps) \cE_\Omega \cL_0^{-1} \cE_\Omega + C \eps^{-q}\cE_\Omega \cL_0^{-2} \cE_\Omega  \\
&\le (1+\eps) (-\Delta_{\R^d})^{-s} + C \eps^{-q} (-\Delta_{\R^d})^{-2s} =: A_\eps^{-1} \; \text{on }L^2(\R^d). 
\end{align*}
Therefore, if we define
$$
v_n= A_\eps^{\frac{1}{2}} \mathcal{E}_\Omega u_n \in L^2(\R^d), 
$$
then the condition \eqref{eq:exclusion-Lv-intro} implies that 
\begin{align} \label{eq:exc-vn-V}
\sum_{n=1}^N |v_n\rangle \langle v_n| \le 1 \quad \text{ on }L^2(\R^d). 
\end{align}
Using \eqref{eq:exc-vn-V}, Bessel's inequality and Plancherel's theorem,  we can bound 
\begin{align} \label{rho-delta-ell-V} 
\rho^{\delta,\ell}(x) &=\sum_{n=1}^N  |u_n^{\delta,\ell}(x)|^2   = \sum_{n=1}^N   \left| \int_{\R^d} e^{2\pi {\mathrm i} x \cdot \xi} \1_{ \{ \delta   \leq |2\pi \xi|^{2s}  < \ell \} } \widehat{\cE_{\Omega} u_n} (\xi)\, {\rm d}\xi \right|^2\nn\\
&=\sum_{n=1}^N   \left| \int_{\R^d} e^{2\pi {\mathrm i} x \cdot \xi} \1_{ \{ \delta   \leq |2\pi \xi|^{2s}  < \ell \} } \sqrt{(1+\eps) |2\pi \xi|^{-2s} + C \eps^{-q} |2\pi \xi|^{-4s} } \widehat{v}_n (\xi) {\rm d}\xi \right|^2 \nn\\
&\le   \int_{\R^d}  \1_{ \{ \delta   \leq |2\pi \xi|^{2s}  < \ell \} } \left( (1+\eps) |2\pi \xi|^{-2s} + C \eps^{-q} |2\pi \xi|^{-4s} \right) {\rm d}\xi \nn\\
&\le (1+\eps) \frac{\omega_d}{(2\pi)^d} \ln \left(\frac{\ell}{\delta} \right) + C \eps^{-q} \delta^{-1}. 
\end{align}

Now let us conclude. Let $\eps \in (0,1/10]$ and choose $\delta=1$. Inserting \eqref{est:sum-un-<delta-frac-V} and \eqref{rho-delta-ell-V} in \eqref{eq:final-thm-1}, and using $\sqrt{a+b}\le \sqrt{a}+\sqrt{b}$ for $a,b\ge 0$, we deduce from \eqref{eq:final-thm-0} that 
\begin{align}\label{eq:final-thm-2}
	&N + C \eps^{-q} \int_{\Omega} \rho(x) \, {\mathrm d}x \nn\\
	& \ge (1-\eps) \int_{\Omega} \int_1^\infty   \left[ \sqrt{\rho(x)} - c\, \eps^{-\frac{q}{2}} - \sqrt{(1+\eps) \frac{\omega_d}{(2\pi)^d} \ln \left(\ell \right)}  \right]_+^2 {\rm d}\ell \, {\rm d}x,
\end{align}
where $c=c(d,p,\Omega,V)>0$. 
We can estimate further the right-hand side of \eqref{eq:final-thm-2} by restricting to the domain $\{x \in \Omega: \rho(x)\geq c^2  \eps^{-q-2}\}$, where
\begin{align*}
	&\int_1^\infty   \left[ \sqrt{\rho(x)} - c\, \eps^{-\frac{q}{2}} - \sqrt{(1+\eps) \frac{\omega_d}{(2\pi)^d} \ln \left(\ell \right)}  \right]_+^2 {\rm d}\ell \\
	&\ge \int_1^\infty   \left[ (1-\eps) \sqrt{\rho(x)} - \sqrt{(1+\eps) \frac{\omega_d}{(2\pi)^d} \ln \left(\ell \right)}  \right]_+^2 {\rm d}\ell \\
	&\geq C \eps^{2} \exp\left( \frac{(2\pi)^d}{\omega_d} (1-5\eps)\rho(x) \right). 
\end{align*}
Therefore, \eqref{eq:final-thm-2} gives 
$$
N + C \eps^{-q} \int_{\Omega} \rho(x)\, {\mathrm d}x \geq C^{-1} \eps^{2} \int_{ \{x\in \Omega: \rho(x)\geq  c^2\eps^{-q-2} \}}  \exp\left( \frac{(2\pi)^d}{\omega_d} (1-5\eps) \rho(x) \right)  {\rm d}x .
$$
Combining with the simple bound $e^{C\rho} \leq e^{C c^2 \eps^{-q-2}}$ on the domain $\{x \in \Omega: \rho(x)\leq c^2 \eps^{-q-2}\}$, we find that 
\begin{align}\label{eq:final-thm-5}
	\frac{1}{|\Omega|}\int_{\Omega}  \exp\left( \frac{(2\pi)^d}{\omega_d} (1-5\eps) \rho(x) \right)  {\rm d}x \le C N \eps^{-2} +  \exp(C\eps^{-q-2}).
\end{align}
For $N$ large, we can choose
$$
\eps = \frac{1}{5}(\ln(N))^{-t}
$$
for some small constant $t=t(d,p) \in (0,1)$. We deduce from  \eqref{eq:final-thm-5} that
$$
\frac{1}{|\Omega|}\int_{\Omega}  \exp\left( \frac{(2\pi)^d}{\omega_d} (1-(\ln(N))^{-t}) \rho(x) \right)  {\rm d}x \le C N (\ln(N))^q. 
$$
Consequently, for every given constant $\alpha>0$, for $N$ sufficiently large, since 
$$
\beta=\frac{ \alpha \omega_d}{ (2\pi)^d (1-(\ln(N))^{-t}) \ln N} \le 1,$$
by Jensen's inequality, we conclude that 
\begin{align*}
&\frac{1}{|\Omega|}\int_{\Omega}  \exp\left( \alpha\frac{\rho(x)}{\ln N} \right)  {\rm d}x \le \left( \frac{1}{|\Omega|}\int_{\Omega}  \exp\left( \frac{(2\pi)^d}{\omega_d} (1-(\ln(N))^{-t}) \rho(x) \right)  {\rm d}x  \right)^\beta \\
&\le \left( C N (\ln(N))^q \right)^\beta \le \exp \left( \frac{\alpha\omega_d}{(2\pi)^d} \left[ 1+ \frac{C}{(\ln N)^t}\right] \right) 
\end{align*}
as desired. The proof of  Theorem \ref{thm:main-Sch} is complete.
\end{proof}

\end{document}